\documentclass[a4paper,11pt]{article}

\usepackage[square,numbers]{natbib}
\usepackage{amsfonts,amsmath,amssymb,amsthm,amsxtra,amsopn}
\usepackage{hyperref}
\usepackage{mathpazo}
\usepackage[usenames,dvipsnames,table]{xcolor}
\usepackage{tikz}
\usepackage{blkarray}
\usepackage{mathdots}
\usepackage{paralist}
\usepackage{listings}

\usetikzlibrary{arrows}
\usetikzlibrary{decorations}
\usetikzlibrary{shapes}

\theoremstyle{plain}
\newtheorem{Theorem}{Theorem}[section]		
\newtheorem{Lemma}[Theorem]{Lemma}

\theoremstyle{definition}

\theoremstyle{remark}

\newtheorem{Example}[Theorem]{Example}

\def\Big#1{\makebox(0,0){\huge#1}}

\newcommand{\mat}[3]{\operatorname{Mat}_{#1\times #2}(\mathbb{#3})}

\title{Quantisation Spaces of Cluster Algebras}

\author{Florian Gellert and Philipp Lampe
	\thanks{florian.gellert@math.uni-bielefeld.de, lampe@math.uni-bielefeld.de}
}
\date{\today}
\begin{document}
\maketitle
\begin{abstract}

The article concerns the existence and uniqueness of quantisations of cluster algebras. 
We prove that cluster algebras with an initial exchange matrix of full rank admit a quantisation 
in the sense of Berenstein-Zelevinsky and give an explicit generating set to construct all quantisations.

\end{abstract}

%%%%%%%%%%%%%%%%%%%%%%%%%%%%%%%%%%%%%%%%%%%%%%%%%%%%%%%%%%
%%%%%%%%%%%%%%%%%%%%%%%%%%%%%%%%%%%%%%%%%%%%%%%%%%%%%%%%%%
\section{Introduction}
%%%%%%%%%%%%%%%%%%%%%%%%%%%%%%%%%%%%%%%%%%%%%%%%%%%%%%%%%%
%%%%%%%%%%%%%%%%%%%%%%%%%%%%%%%%%%%%%%%%%%%%%%%%%%%%%%%%%%

Sergey Fomin and Andrei Zelevinsky have introduced and studied cluster algebras in a series of four
 articles~\cite{FZ1,FZ2,BFZ3,FZ4} (one of which is coauthored by Arkady Berenstein) in order
 to study Lusztig's canonical basis and total positivity. Cluster algebras are commutative algebras which are 
constructed by generators and relations. The generators are called cluster variables and they are grouped
 into several overlapping sets, so-called clusters. A combinatorial mutation process relates the clusters and
 provides the defining relations of the algebra. The rules for this mutation process are encoded 
in a rectangular exchange matrix usually denoted by the symbol $\tilde{B}$. 
Surprinsingly, Fomin-Zelevinsky's Laurent phenomenon~\cite[Thm.~3.1]{FZ1} asserts that 
every cluster variable can be expressed as a Laurent polynomial in an arbitrarily chosen cluster. 
The second main theorem of cluster theory is the classification of cluster algebras with only finitely 
many cluster variables by finite type root systems, see Fomin-Zelevinsky~\cite{FZ2}.

The connections of cluster algebras to other areas of mathematics are manifold. 
A major contribution is Caldero-Chapoton's map~\cite{CC} which relates cluster algebras and 
representation theory of quivers. Another contribution is the construction of cluster algebras 
from oriented surfaces which relates cluster algebras and differential geometry, 
see Fomin-Shapiro-Thurston~\cite{FST1} and Fomin-Thurston~\cite{FT2}.

Arkady Berenstein and Andrei Zelevinsky~\cite{BZ} have introduced the concept of quantum cluster algebras. 
Quantum cluster algebras are $q$-deformations which specialise to the ordinary 
cluster algebras in the classical limit $q=1$. Such quantisations play an important role in cluster theory: 
on the one hand, quantisations are essential when trying to link cluster algebras to Lusztig's canonical bases, 
see for example~\cite{Lu,Le,La1,La2,GLS,HL}. 
On the other hand, Goodearl-Yakimov~\cite{GY} use quantisations to approximate cluster algebras 
by their upper bounds. 
The latter result is particularly important since it enables us to study cluster algebras as unions of 
Laurent polynomial rings.

In general, the notion of $q$-deformation turns former commutative structures into 
non-commutative ones. In the case of quantum cluster algebras, this yields $q$-commutativity between 
variables within the same quantum cluster which is stored in an additional matrix usually denoted by the 
symbol $\Lambda$. In order to keep the $q$-commutativity intact under mutation, 
Berenstein and Zelevinsky require some compatibility relation between
the matrices $\tilde{B}$ and $\Lambda$. 
The very same compatibility condition also parametrises compatible Poisson structures for 
cluster algebras, see Gekhtman-Shapiro-Vainshtein~\cite{GSV}. 

Unfortunately, not every cluster algebra admits a quantisation, 
because not every exchange matrix admits a compatible $\Lambda$. 
But in the case where there exists such a quantisation, 
Berenstein-Zelevinsky have shown that $\tilde{B}$ is of full rank. 

This paper has several aims. Firstly, we reinterpret what it means for quadratic exchange matrices 
to be of full rank via Pfaffians and perfect matchings. 
Secondly, we show the converse of the above statement 
(which was posted as a question during Zelevinsky's lecture at a workshop): 
assuming $\tilde{B}$ is of full rank, there always exists a quantisation. 
This result we show by using concise linear algebra arguments. 
It should be noted that Gekhtman-Shapiro-Vainshtein in~\cite[Thm.~4.5]{GSV} prove a similar statement 
in the language of Poisson structures. 
Thirdly, when a quantisation exists, it is not necessarily unique. 
This ambiguity we make more precise by relating all such quantisations 
via matrices we construct from a given $\tilde{B}$ using particular minors. 

%%%%%%%%%%%%%%%%%%%%%%%%%%%%%%%%%%%%%%%%%%%%%%%%%%%%%%%%%%
%%%%%%%%%%%%%%%%%%%%%%%%%%%%%%%%%%%%%%%%%%%%%%%%%%%%%%%%%%
\section{Berenstein-Zelevinsky's quantum cluster algebras}
%%%%%%%%%%%%%%%%%%%%%%%%%%%%%%%%%%%%%%%%%%%%%%%%%%%%%%%%%%
%%%%%%%%%%%%%%%%%%%%%%%%%%%%%%%%%%%%%%%%%%%%%%%%%%%%%%%%%%

%%%%%%%%%%%%%%%%%%%%%%%%%%%%%%%%%%%%%%%%%%%%%%%%%%%%%%%%%%
\subsection{Notation}
%%%%%%%%%%%%%%%%%%%%%%%%%%%%%%%%%%%%%%%%%%%%%%%%%%%%%%%%%%

Let $m,n$ be integers with $1 \leq n \leq m$ and 
$A$ an $m\times n$ matrix with integer entries.
For $n<m$ we use the notation $[n,m] = \{n+1, \ldots, m\}$ and in the case $n=1$ the 
shorthand $[m]=[1, m]$. 
For a subset $J \subseteq [m]$ we denote by $A_J$ the submatrix of $A$
with rows indexed by $J$ and all columns. By $q$ we denote throughout the paper a formal
indeterminate.

%%%%%%%%%%%%%%%%%%%%%%%%%%%%%%%%%%%%%%%%%%%%%%%%%%%%%%%%%%
\subsection{The definition of quantum cluster algebras}
%%%%%%%%%%%%%%%%%%%%%%%%%%%%%%%%%%%%%%%%%%%%%%%%%%%%%%%%%%

Let $\tilde{B}=(b_{i,j})$ be an $m\times n$ matrix 
with integer entries. For further use we write 
$\tilde{B}=\left[\begin{smallmatrix}B\\C\\\end{smallmatrix}\right]$
in block form with an $n\times n$ matrix $B$ and an $(m-n)\times n$ matrix $C$. 
The matrix $B$ is called the \emph{principal part} of $\tilde{B}$. 
We call indices $i\in [n]$ \emph{mutable} and the indices $j\in [n+1,m]$ \emph{frozen}. 

We say that the principal part $B$ is \emph{skew-symmetrisable} if there exists a diagonal $n\times n$ matrix 
$D=\textrm{diag}(d_1,d_2,\ldots,d_n)$ with positive integer diagonal entries such that the matrix 
$DB$ is skew-symmetric, i.\,e. $d_ib_{i,j}=-d_jb_{j,i}$ for all $1\leq i,j\leq n$. 
The matrix $D$ is then called a \emph{skew-symmetriser} for $B$ and $\tilde{B}$ is called an 
\emph{exchange matrix}. Note that skew-symmetrising from the right yields identical restraints and
$b_{i,j}\neq 0$ if and only if $b_{j,i}\neq 0$. 

The skew-symmetriser is essentially unique by the following discussion:
Consider  the unoriented simple graph $\Delta(B)$ with 
vertex set $\{1,2,\ldots ,n\}$ such that there is an edge between two vertices $i$ and $j$ if and only if $b_{ij}\neq 0$. 
We say that the principal part $B$ is \emph{connected} if $\Delta(B)$ is connected. 
Note that the connectedness of the principal part is mutation invariant, i.\,e. 
if $B$ is connected, then the principal part of $\mu_k(\tilde{B})$ is connected for all $1\leq k\leq n$ as well. 

Assume now that $B$ is connected. 
Suppose there exist two diagonal $n\times n$-matrices
 $D$ and $D'$ with positive integer diagonal entries such that both $DB$ and $D'B$ are skew-symmetric. 
Then there exists a rational number $\lambda$ with $D=\lambda D'$, 
as for all indices $i,j$ with $b_{ij}\neq 0$ the equality $d_i/d_j=d'_i/d'_j$ holds true. 
We refer to the smallest such $D$ as the \emph{fundamental skew-symmetriser}.   
If $B$ is not connected, then every skew-symmetriser $D$ is a $\mathbb{N}^{+}$-linear combination 
of the fundamental skew-symmetrisers of the connected components of $B$. 

This concludes the discussion of the first datum to construct quantum cluster algebras. 
The next piece of data is the notion of compatible matrix pairs. 

From now on, assume that $\tilde{B}=\left[\begin{smallmatrix}B\\C\\\end{smallmatrix}\right]$ 
is a not necessarily connected matrix with skew-symmetrisable principal part $B$.
A skew-symmetric $m\times m$ integer matrix $\Lambda = (\lambda_{i,j})$ is called \emph{compatible} 
if there exists a diagonal $n\times n$ matrix $D'=\textrm{diag}(d'_1,d'_2,\ldots,d'_n)$ with 
positive integers $d'_1,d'_2,\ldots,d'_n$ such that 
\begin{equation}\label{eq:comp}
  \tilde{B}^T\Lambda =\left[\begin{matrix}D'&0\end{matrix}\right]
\end{equation}
as a $n\times n$ plus $n\times (m-n)$ block matrix. 
In this case we call $(\tilde{B}, \Lambda)$ a \emph{compatible} pair.
To any $m\times n$ matrix $\tilde{B}$ there need not exist a compatible $\Lambda$. 
As a necessary condition Berenstein-Zelevinsky~\cite[Prop.~3.3]{BZ} note that if a matrix $\tilde{B}$ 
belongs to a compatible pair $(\tilde{B}, \Lambda)$, then its principal part $B$ is skew-symmetrisable, 
$D'$ itself is a skew-symmetriser and $\tilde{B}$ itself is of full rank, i.\,e. $\textrm{rank}(\tilde{B})=n$.

Let us fix a compatible pair $(\tilde{B},\Lambda)$.
We are now ready to complete the necessary data to define  quantum cluster algebras. 
First of all, let $\{e_i\colon 1\leq 1\leq m\}$ be the standard basis of $\mathbb{Q}^m$. 
With respect to this standard basis the skew-symmetric matrix $\Lambda$ defines a 
skew-symmetric bilinear form $\beta\colon\mathbb{Q}^m\times \mathbb{Q}^m\to \mathbb{Q}$. 
The \emph{based quantum torus} $\mathcal{T}_{\Lambda}$ associated with $\Lambda$ 
is the $\mathbb{Z}[q^{\pm 1}]$-algebra with $\mathbb{Z}[q^{\pm 1}]$-basis $\{X^a\colon a\in\mathbb{Z}^m\}$ 
where we define the multiplication of basis elements by the formula 
$X^aX^b=q^{\beta(a,b)}X^{a+b}$ for all elements $a,b\in\mathbb{Z}^m$. 
It is an associative algebra with unit $1=X^0$ and every basis element $X^a$ has an inverse $(X^a)^{-1}=X^{-a}$. 
The based quantum torus is commutative if and only if $\Lambda$ is the zero matrix, 
in which case $\mathcal{T}_{\Lambda}$ is a Laurent polynomial algebra. 
In general, it is an Ore domain, see~\cite[Appendix]{BZ} for further details.
We embed $\mathcal{T}_{\Lambda}\subseteq \mathcal{F}$ into an ambient skew field.

Although $\mathcal{T}_{\Lambda}$ is not commutative in general, the relation $X^aX^b=q^{2\beta(a,b)}X^bX^a$ 
holds for all elements $a,b\in\mathbb{Z}^m$. 
Because of this relation we say that the basis elements are $q$-\emph{commutative}. 
Put $X_i=X^{e_i}$ for all $i\in[m]$. The definition implies $X_iX_j=q^{\lambda_{i,j}}X_jX_i$ for all $i,j\in[m]$. 
Then we may write 
$\mathcal{T}_{\Lambda}=\mathbb{Z}[q^{\pm 1}][X_1^{\pm 1},X_2^{\pm 1},\ldots,X_m^{\pm 1}]$ 
and the basis vectors satisfy the relation
\begin{align*}
X^{a}=q^{\sum_{i>j}\lambda_{i,j}a_ia_j}X_1^{a_1}X_2^{a_2}\cdot\ldots\cdot X_m^{a_m}
\end{align*}
for all $a=(a_1,a_2,\ldots,a_m)\in\mathbb{Z}^m$. 

We call a sequence of pairwise $q$-commutative and algebraically independent elements 
such as $X=(X_1,X_2,\ldots,X_m)$ in $\mathcal{F}$ an \emph{extended quantum cluster}, 
the elements $X_1,X_2,\ldots,X_n$ of an extended quantum cluster \emph{quantum cluster variables}, 
the elements $X_{n+1},X_{n+2},\ldots,X_m$ \emph{frozen variables} and 
the triple $(\tilde{B},X,\Lambda)$ a \emph{quantum seed}.

Let $k$ be a mutable index. 
Define mutation map $\mu_k\colon (\tilde{B},X,\Lambda) \mapsto (\tilde{B}',X',\Lambda')$ as follows: 
\begin{enumerate}[($M_1$)]
	\item The matrix $\tilde{B}' = \mu_k(\tilde{B})$ the well-known 
		mutation of skew-symmetrisable matrices.
	\item The matrix $\Lambda' = (\lambda_{i,j}')$ is the $m\times m$ matrix with entries 
		$\lambda_{i,j}'=\lambda_{i,j}$ except for
		\[
			\begin{aligned}
				\lambda_{i,k}' &= -\lambda_{i,k} + \sum_{r\neq k} \lambda_{i,r} \max(0,-b_{r,k}) \text{ for all } i \in [m]\backslash \{k\},\\
				\lambda_{k,j}' &= -\lambda_{k,j} - \sum_{r\neq k} \lambda_{j,r} \max(0,-b_{r,k}) \text{ for all } j \in [m]\backslash \{k\}.
		\end{aligned}
		\]
	\item To obtain the quantum cluster $X'$ we replace the element $X_k$ with the element 
		\[
			X'_k=X^{-e_k+\sum \limits_{i=1}^{m} \max(0,b_{i,k})e_i}+X^{-e_k+\sum \limits_{i=1}^{m} \max(0,-b_{i,k})e_i}\in\mathcal{F}.
		\]
\end{enumerate}
The variables $X'=(X_1',X_2',\ldots,X_m')$ are pairwise $q$-commutative: for all $j\in[m]$ with $j\neq k$ the integers
\begin{align*}
	&\beta\left(-e_k+\sum_{i=1}^m\max(0,b_{i,k})e_i,e_j\right)=-\lambda_{k,j}+\sum_{i=1}^m\max(0,b_{i,k})\lambda_{i,j}\\
	&\beta\left(-e_k+\sum_{i=1}^m\max(0,-b_{i,k})e_i,e_j\right)=-\lambda_{k,j}+\sum_{i=1}^m\max(0,-b_{i,k})\lambda_{i,j}
\end{align*}
are equal, because their difference is equal to the sum $\sum_{i=1}^m b_{i,k}\lambda_{i,j}$ 
which is the zero entry indexed by $(k,j)$ in the matrix $\tilde{B}^T\Lambda$. 
So the variable $X_k'$ $q$-commutes with all $X_j$. 
Hence the variables $X'=(X'_1,X'_2,\ldots,X'_m)$ generate again a based quantum torus 
whose $q$-commutativity relations are given by the skew-symmetric matrix $\Lambda'$. 
Moreover, the pair $(\tilde{B}',\Lambda')$ is compatible by~\cite[Prop. 3.4]{BZ} so that 
the matrix $\tilde{B}'$ has a skew-symmetrisable principle part. 
We conclude that the mutation $\mu_k(\tilde{B}',X',\Lambda')=(\tilde{B}',X',\Lambda')$ is again an extended quantum seed. 
Note that the mutation map is involutive, i.\,e. $(\mu_k\circ\mu_k)(\tilde{B},X,\Lambda)=(\tilde{B},X,\Lambda)$. 

Here we see the importance of the compatibility condition. 
A main property of classical cluster algebras are the binomial exchange relations. 
For the quantised version we require pairwise $q$-commutativity for the quantum cluster variables in a single cluster. 
This implies that a monomial $X_1^{a_1}X_2^{a_2}\cdot\ldots\cdot X_m^{a_m}$ with $a\in\mathbb{Z}^m$
remains (up to a power of $q$) a monomial under reordering the quantum cluster variables. 

We call two quantum seeds $(\tilde{B},X,\Lambda)$ and $(\tilde{B}',X',\Lambda')$ \emph{mutation equivalent} 
if one can relate them by a sequence of mutations. 
This defines an equivalence relation on quantum seeds, denoted by $ (\tilde{B},X,\Lambda) \sim (\tilde{B}',X',\Lambda')$. 
The \emph{quantum cluster algebra} $\mathcal{A}_q(\tilde{B},X,\Lambda)$ associated to a given quantum seed 
$(\tilde{B},X,\Lambda)$ is the $\mathbb{Z}[q^{\pm 1}]$-subalgebra of $\mathcal{F}$ generated by the set 
\[
	\chi(\tilde{B},X,\Lambda) = \left\{ X_{i}^{\pm 1}\;|\; i \in [n+1,m] \, \right\} \cup 
		\bigcup_{ (\tilde{B}',X',\Lambda') \sim (\tilde{B},X,\Lambda)} \left \{X_i' \;|\; i\in [n] \,\right \}.
\]

The specialisation at $q=1$ identifies the quantum cluster algebra $\mathcal{A}_q(\tilde{B},X,\Lambda)$ with 
the classical cluster algebra $\mathcal{A}(\tilde{B},X)$. Generally, the definitions of classical and 
quantum cluster algebras admit additional analogies. 
One such analogy is the quantum Laurent phenomenon, 
as proven in~\cite[Cor.~5.2]{BZ}: $\mathcal{A}_q(\tilde{B},X,\Lambda) \subseteq \mathcal{T}_\Lambda$.
Remarkably, $\mathcal{A}_q(\tilde{B},X,\Lambda)$ and $\mathcal{A}(\tilde{B},X)$ also possess 
the same exchange graph by~\cite[Thm.~6.1]{BZ}. 
In particular, quantum cluster algebras of finite type are also classified by Dynkin diagrams.

%%%%%%%%%%%%%%%%%%%%%%%%%%%%%%%%%%%%%%%%%%%%%%%%%%%%%%%%%%
%%%%%%%%%%%%%%%%%%%%%%%%%%%%%%%%%%%%%%%%%%%%%%%%%%%%%%%%%%
\section{The quantisation space}
%%%%%%%%%%%%%%%%%%%%%%%%%%%%%%%%%%%%%%%%%%%%%%%%%%%%%%%%%%
%%%%%%%%%%%%%%%%%%%%%%%%%%%%%%%%%%%%%%%%%%%%%%%%%%%%%%%%%%

\subsection{Remarks on skew-symmetric matrices of full rank}

How can we decide whether an exchange matrix $\tilde{B}$ has full rank? 
Let us consider the case $n=m$. 
Multiplication with a skew-symmetriser $D$ does not change the rank, 
so without loss of generality we may assume that $\tilde{B}=B$ is skew-symmetric. 
In this case, $B=B(Q)$ is the signed adjacency matrix of some quiver $Q$ with $n$ vertices.  

First of all, if $n$ is odd, then $B$ can not be of full rank, because $\det(B)=(-1)^n\det(B)$ implies $\det(B)=0$. 
Especially, no (coefficient-free) cluster algebra attached to a quiver $Q$ with an odd number of vertices admits a quantisation.

Now suppose that $n=m$ is even. 
In this case, a theorem of Cayley \cite{C} asserts that there exists a polynomial $\operatorname{Pf}(B)$ 
in the entries of $B$ such that $\det(B)=\operatorname{Pf}(B)^2$. The polynomial is called the \emph{Pfaffian}. 
For example, if $n=4$, then $\operatorname{Pf}(B)=b_{12}b_{34}-b_{13}b_{24}+b_{14}b_{23}$. 
For general $n$ we have
\[
	\operatorname{Pf}(B)=\sum \operatorname{sgn}(i_1,\ldots,i_{n/2},j_1,\ldots,j_{n/2}) b_{i_1j_1}b_{i_2j_2}\cdots b_{i_{n/2}j_{n/2}}
\]
where the sum is taken over all $(n-1)(n-3)\cdots1$ possibilities of writing the set $\{1,2,\ldots,n\}$ as 
a union $\{i_1,j_2\}\cup\{i_2,j_2\}\cup\ldots \cup\{i_{n/2},j_{n/2}\}$ of $\frac n2$ sets of cardinality $2$ and 
$\operatorname{sgn}(i_1,\ldots,i_{n/2},j_1,\ldots,j_{n/2})\in\{\pm 1\}$ is the sign of the permutation
 $\sigma\in S_n$ with $\sigma(2k-1)=i_k$ and $\sigma(2k)=j_k$ for all $k\in\{1,2,\ldots,\frac{n}{2}\}$, 
cf. Knuth~\cite[Equation (0.1)]{K}. It is easy to see that the sum is well-defined.

Note that in the sum above a summand vanishes unless 
$\{i_1,j_2\},\{i_2,j_2\},\ldots,\{i_{n/2},j_{n/2}\}$ is a \emph{perfect matching} 
of the underlying undirected graph of $Q$. 

For example, let $Q=\overrightarrow{A_n}$ be an orientation of a Dynkin diagram of type $A_n$ with an even number $n$. 
Then $Q$ admits exactly one perfect matching $\{1,2\},\{3,4\},\ldots,\{n-1,n\}$. Hence, $\det(B(Q))=\pm 1$ so that $B(Q)$ is regular. 
The same is true for all quivers $Q$ of type $\overrightarrow{E_6}$ or $\overrightarrow{E_{8}}$. 
On the other hand, there does not exist a perfect matching for a Dynkin diagram of type $D_n$. 
Hence, $\det(B(Q))=0$ for all quivers $Q$ of type $D_n$. 

To summarize, a (coefficient-free) cluster algebra of finite type has a quantization if and only if 
it is of Dynkin type $A_n$ with even $n$ or of type $E_6$ or $E_8$. 
(These are precisely the Dynkin diagrams for which the stable category 
$\underline{\operatorname{CM}}(R)$ of Cohen-Macaulay modules of the corresonding 
hypersurface singularity $R$ of dimension $1$ does not have an indecomposable rigid object, 
see Burban-Iyama-Keller-Reiten~\cite[Theorem 1.3]{BIKR}.)          

%%%%%%%%%%%%%%%%%%%%%%%%%%%%%%%%%%%%%%%%%%%%%%%%%%%%%%%%%%
\subsection{Existence of quantisation}\label{subsec:existence}
%%%%%%%%%%%%%%%%%%%%%%%%%%%%%%%%%%%%%%%%%%%%%%%%%%%%%%%%%%

Suppose that $\operatorname{rank}(\tilde{B})=n$. 
In this subsection we prove that the cluster algebra $\mathcal{A}(\tilde{\mathbf{x}},\tilde{B})$ admits a quantisation. 

The $n$ column vectors of $\tilde{B}$ are linearly independent elements in $\mathbb{Q}^m$. 
We extend them to a basis of $\mathbb{Q}^m$ by adding $(m-n)$ appropriate column vectors. 
Hence, there is an invertible $m \times n$ plus $(m-n)\times m$ block matrix 
$\left[\begin{matrix}\tilde{B}&\tilde{E}\end{matrix}\right]\in\operatorname{GL}_m(\mathbb{Q})$ 
which we denote by $M$. We also write $\tilde{E}$ itself in block form as
$\tilde{E}=\left[\begin{smallmatrix}E\\F\\\end{smallmatrix}\right]$
with an $n\times (m-n)$ matrix $E$ and an $(m-n)\times (m-n)$ matrix $F$. 
Of course, the choice for the basis completion is not canonical. 
In particular, one can choose standard basis vectors for columns of $\tilde{E}$, making it sparse.
After these preparations we are ready to state the theorem about the existence of a quantization.

\begin{Theorem}\label{thm:existence}\label{Quantizations}
	Let $D$ be a skew-symmetriser of $B$. 
	There exists a skew-symmetric $m\times m$-matrix $\Lambda$ with integer coefficients 
	and a multiple $D'=\lambda D$ with $\lambda\in\mathbb{Q}^{+}$ such that
	 $\tilde{B}^T\Lambda =\left[\begin{matrix}D'&0\end{matrix}\right]$.
\end{Theorem}

\begin{proof} 
	Put 
	\[
		\Lambda_0=M^{-T}\left[
			\begin{matrix}
				DB&DE\\
				-E^TD&0\\
			\end{matrix}\right]M^{-1}\in\operatorname{Mat}_{m\times m}(\mathbb{Q})
	\]
	and let $\Lambda$ be the smallest multiple of $\Lambda_0$ which lies in
	 $\operatorname{Mat}_{m\times m}(\mathbb{Z})$. 
	The matrix $\Lambda$ is skew-symmetric by construction and the relation $M^TM^{-T}=I_{m,m}$ implies
	 $\tilde{B}^TM^{-T}=\left[\begin{matrix}I_{n,n}&0_{n,m-n}\end{matrix}\right]$. Thus,
	\[
		\tilde{B}^T\Lambda_0=\left[\begin{matrix}DB&DE\\\end{matrix}\right]M^{-1}=
		D\left[\begin{matrix}B&E\\\end{matrix}\right]M^{-1}=D\left[\begin{matrix}I_{n,n}&0_{n,m-n}\\\end{matrix}\right]=
		\left[\begin{matrix}D&0\end{matrix}\right].
	\]
	Scaling the equation yields $\tilde{B}^T\Lambda=\left[\begin{matrix}D'&0\end{matrix}\right]$ 
	for some multiple $D'$ of $D$.
\end{proof}

Together with Berenstein-Zelevinsky's initial result this means that a cluster algebra 
$\mathcal{A}(\tilde{B})$ admits a quantisation if and only if $\tilde{B}$ has full rank. 
Since the rank of the exchange matrix is mutation invariant, 
one can use any seed to check whether a cluster algebra admits a quantisation.

Zelevinsky~\cite{Z} suggests to reformulate the statement in terms of bilinear forms. 
With respect to the standard basis, the matrix $\Lambda$ defines a skew-symmetric bilinear form. 
Let us change the basis. The column vectors $b_1,b_2,\ldots,b_n$ of $\tilde{B}$ are linearly independent over $\mathbb{Q}$. 
Let  $V'=\operatorname{span}_{\mathbb{Q}}(b_1,b_2,\ldots,b_n)$ be the column space of $\tilde{B}$. 
The column vectors $\tilde{e}_{n+1},\tilde{e}_{n+2},\ldots,\tilde{e}_m$ of $\tilde{E}$ extend to a basis of $V=\mathbb{Q}^m$. 
Let $V''=\operatorname{span}_{\mathbb{Q}}(\tilde{e}_{n+1},\tilde{e}_{n+2},\ldots,\tilde{e}_{m})$. 
The compatibility condition $\tilde{B}^T\Lambda =\left[\begin{matrix}D'&0\end{matrix}\right]$ says 
that for any given $D'$, the skew-symmetric bilinear form $V\times V\to \mathbb{Q}$ is completely determined on $V'\times V$, 
hence also on $V\times V'$.
Such a bilinear form can be chosen freely on $V''\times V''$ giving a $\frac{1}{2}(m-n-1)(m-n-2)$-dimensional solution space. 
In particular, the quantisation is essentially unique (i.\,e. unique up to a scalar) when there are only $0$ or $1$ frozen vertices. 

%%%%%%%%%%%%%%%%%%%%%%%%%%%%%%%%%%%%%%%%%%%%%%%%%%%%%%%%%%
%%%%%%%%%%%%%%%%%%%%%%%%%%%%%%%%%%%%%%%%%%%%%%%%%%%%%%%%%%
\section{A minor generating set}\label{sec:minor}
%%%%%%%%%%%%%%%%%%%%%%%%%%%%%%%%%%%%%%%%%%%%%%%%%%%%%%%%%%
%%%%%%%%%%%%%%%%%%%%%%%%%%%%%%%%%%%%%%%%%%%%%%%%%%%%%%%%%%

In the previous section we observed that any full-rank
skew-symmetrisable matrix $\tilde{B}$ admits a quantisation. 
In the construction yielding Theorem~\ref{thm:existence}, 
we chose some $m \times (m-n)$
integer matrix $\tilde{E}$ which completed a basis for 
$\mathbb{Q}^m$. This choice we now reformulate by giving
a generating set of integer matrices for the equation
\begin{equation}\label{eq:homog}
	\tilde{B}^T\Lambda = \left[\begin{matrix}0 &0\end{matrix}\right].
\end{equation}

As previously remarked, this ambiguity does not occur for
$0$ or $1$ frozen vertices, hence we may start 
with the case $m=n+2$ in Section~\ref{subsec:two}. 
From this result we construct 
such a generating set for arbitrary $m$ with $|m-n|>2$ in the subsequent section.
The construction below holds in more generality than what is naturally required in our setting.
Thus we now consider an arbitrary integer matrix $A$ of dimension $m\times n$ instead of $\tilde{B}$ 
and obtain the generating set for equation~\eqref{eq:homog} as a consequence.

%%%%%%%%%%%%%%%%%%%%%%%%%%%%%%%%%%%%%%%%%%%%%%%%%%%%%%%%%%
\subsection{Minor blocks} \label{subsec:two}
%%%%%%%%%%%%%%%%%%%%%%%%%%%%%%%%%%%%%%%%%%%%%%%%%%%%%%%%%%

In this subsection we assume $m = n+2$.

For distinct $i,j \in [m]$ define a reduced indexing set
$R(i,j)$ as the $n$-element subset of $[m]$ in which $i$ and $j$ do not occur.
To an arbitrary $m \times n$ integer matrix  $A = (a_{i,j})$ we associate 
the skew-symmetric $m\times m$-integer 
matrix $M = M(A) = \left( m_{i,j} \right)$ with entries
\begin{equation} \label{eq:defm}
	m_{i,j} = \begin{cases}
		(-1)^{i+j} \cdot \det ( A_{R(i,j)} ), & i < j,\\
		0, & i = j, \\
		(-1)^{i+j+1} \cdot \det ( A_{R(i,j)} ), & j < i.
	\end{cases}
\end{equation}

Then we first observe the following property of $M$, which carries some similarity to the well-known Pl\"ucker relations.

\begin{Lemma}\label{prop:sol}
	For $A$ an $m \times n$ integer matrix, we obtain
	\[
		A^T \cdot M = [ 0 \; 0 ].
	\]
\end{Lemma}
\begin{proof}
	By definition, we have
	\[
		\left[ A^T \cdot M \right]_{i,j} = \sum_{k=1}^m a_{k,i} m_{k,j} \notag \\
		= \sum_{k\in [m]\backslash\{j\}} a_{k,i} m_{k,j}. 
	\]
	Now let $A_j$ be the matrix we obtain from $A$ by removing the $j$-th row and 
	$A_j^i$ the matrix that results from attaching the $i$-th column of
	$A_j$ to itself on the right. Then $\det(A_j^i) = 0$ and we observe that 
	using the Laplace expansion along the last column, we obtain the right-hand side 
	of the above equation up to sign change. The claim follows.
\end{proof}

\begin{Example}
	Let $\alpha, a,b,c$ and $d$ be positive integers. Then consider the quiver $Q$ given by:
	
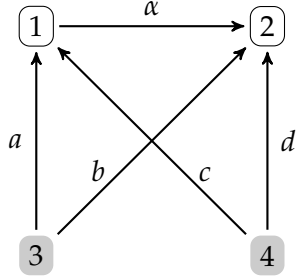
\begin{center}
	\begin{tikzpicture}[scale=0.76] 
		\node[rounded corners, draw] at (0,0) (v1) {$1$};
		\node[rounded corners, draw] at (4,0) (v2) {$2$};
		\node[rounded corners, fill=black!20] at (0,-4) (v3) {$3$};
		\node[rounded corners, fill=black!20] at (4,-4) (v4) {$4$};
		\path[->,thick,shorten <=2pt,shorten >=2pt,>=stealth'] (v1) edge node [above] {$\alpha$} (v2);
		\path[->,thick,shorten <=2pt,shorten >=2pt,>=stealth'] (v3) edge node [left=1pt] {$a$} (v1);
		\path[->,thick,shorten <=2pt,shorten >=2pt,>=stealth'] (v3) edge node [left=20pt, below=3pt] {$b$} (v2);
		\path[->,thick,shorten <=2pt,shorten >=2pt,>=stealth'] (v4) edge node [right=20pt, below=6pt] {$c$} (v1);
		\path[->,thick,shorten <=2pt,shorten >=2pt,>=stealth'] (v4) edge node [right=1pt] {$d$} (v2);
	\end{tikzpicture}
	\end{center}
	Thus the matrices $\tilde{B}$ and $M$ are 
	\[
		\tilde{B} = \begin{bmatrix}
			0 & \alpha\\
			-\alpha & 0 \\
			a & b \\
			c & d
		\end{bmatrix}, \qquad 
		M = \begin{bmatrix}
			0 & - ad + bc & -\alpha d & \alpha b \\
			ad - bc & 0 & \alpha c & -\alpha a\\
			\alpha d & -\alpha c & 0 &  -\alpha^2 \\
			-\alpha b & \alpha a & \alpha^2 & 0
		\end{bmatrix},
	\]
	and we immediately see the result of the previous lemma, namely $\tilde{B}^T \cdot M = [ 0 \; 0].$
\end{Example}

%%%%%%%%%%%%%%%%%%%%%%%%%%%%%%%%%%%%%%%%%%%%%%%%%%%%%%%%%%
\subsection{Composition of minor blocks} \label{subsec:many}
%%%%%%%%%%%%%%%%%%%%%%%%%%%%%%%%%%%%%%%%%%%%%%%%%%%%%%%%%%

In this section let $n+2<m$ and as before, let 
$A\in \mat{m}{n}{Z}$ be some rectangular integer matrix. 
Choose a subset $F\subset [m]$ of cardinality $n$ and 
obtain a partition of the indexing set $[m]$ of the rows of $A$ as $[m] = F \sqcup R$. 
Note that $|R| = m-n$.
For distinct $i,j \in R$ set the \emph{extended indexing set associated to $i, j$} to be 
\[
	E(i,j) := F\cup \{ i,j\}.
\]
By Lemma~\ref{prop:sol} (after a reordering of rows) and slightly abusing the notation, there exists an $ (n+2) \times (n+2) $ matrix 
$M_{E(i,j)}= \left( m_{r, s} \right) $ such that
\begin{equation}
	A_{E(i,j)}^T \cdot M_{E(i,j)} = [0 \; 0 ]. \label{eq:sol2}
\end{equation}
Now let $\mathfrak{M}_{E(i,j)} = \mathfrak{M}_{E(i,j)}(A) = \left( \mathfrak{m}_{r, s} \right) $ be the 
\emph{enhanced solution matrix associated to $i, j$}, 
the $ m \times m $ integer matrix we obtain from $M_{E(i,j)}$ by 
filling the entries labeled by $E(i,j) \times E(i,j)$ with $M_{E(i,j)}$ 
consecutively and setting all other entries to zero.

\begin{Example}
	Consider the quiver $Q$ with associated exchange matrix $\tilde{B}$ as below:
	
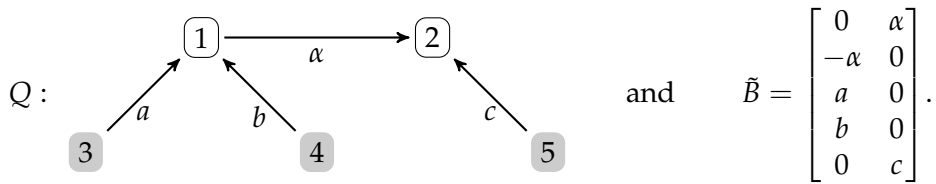
\begin{center}
	\begin{tikzpicture}[scale=0.76] 
		\node at (-1,-1) {$Q:$};
		\node[rounded corners, draw] at (2,0) (v1) {$1$};
		\node[rounded corners, draw] at (6,0) (v2) {$2$};
		\node[rounded corners, fill=black!20] at (0,-2) (v3) {$3$};
		\node[rounded corners, fill=black!20] at (4,-2) (v5-227) {$4$};
		\node[rounded corners, fill=black!20] at (8,-2) (v5-228) {$5$};
		\node at (12,-1) {and $\qquad \tilde{B} = \begin{bmatrix} 0 & \alpha \\ -\alpha & 0 \\ a & 0 \\ b & 0 \\ 0 & c \end{bmatrix} $.};
		\path[->,thick,shorten <=2pt,shorten >=2pt,>=stealth'] (v1) edge node [below] {$\alpha$} (v2);
		\path[->,thick,shorten <=2pt,shorten >=2pt,>=stealth'] (v3) edge node [below] {$a$} (v1);
		\path[->,thick,shorten <=2pt,shorten >=2pt,>=stealth'] (v5-227) edge node [below] {$b$} (v1);
		\path[->,thick,shorten <=2pt,shorten >=2pt,>=stealth'] (v5-228) edge node [below] {$c$} (v2);
	\end{tikzpicture}
	\end{center}
	We choose $F=\{1,2\}$, assuming $\alpha \neq 0$ and get the following matrices $M_{E(i,j)}$ 
	and their  enhanced solution matrices for distinct $i,j \in \{3,4,5\}$:
\begin{align*}
	M_{E(3, 4)} &= \begin{bmatrix} 
		0 & 0 & 0 & 0 \\ 
		0 & 0 & \alpha b & -\alpha a \\ 
		0 & -\alpha b & 0 & -\alpha^2 \\
		0 & \alpha a & \alpha^2 & 0
	\end{bmatrix},  &&\mathfrak{M}_{E(3, 4)} = \left[ \begin{array}{cccc>{\columncolor{black!15}}c} 
		0 & 0 & 0 & 0 & 0\\ 
		0 & 0 & \alpha b & -\alpha a & 0\\ 
		0 & -\alpha b & 0 & -\alpha^2 & 0\\
		0 & \alpha a & \alpha^2 & 0 & 0\\
		\rowcolor{black!15}
		0 & 0 & 0 & 0 & 0
	\end{array} \right], \\
	M_{E(3, 5)} &= \begin{bmatrix} 
		0 & -ac & -\alpha c & 0 \\ 
		ac & 0 & 0 & -\alpha a \\
		\alpha c & 0 & 0 & -\alpha^2 \\
		0 & \alpha a & \alpha^2 & 0
	\end{bmatrix},   && \mathfrak{M}_{E(3, 5)} = \left[ \begin{array}{ccc>{\columncolor{black!15}}cc}
		0 & -ac & -\alpha c & 0 & 0 \\ 
		ac & 0 & 0 & 0 & -\alpha a \\
		\alpha c & 0 & 0 & 0 & -\alpha^2 \\
		\rowcolor{black!15}
		0 & 0 & 0 & 0 & 0\\
		0 & \alpha a & \alpha^2 & 0 & 0
	\end{array} \right], \\
	M_{E(4, 5)} &= \begin{bmatrix} 
		0 & -bc & -\alpha c & 0 \\ 
		bc & 0 & 0 & -\alpha b \\
		\alpha c & 0 & 0 & -\alpha^2 \\
		0 & \alpha b & \alpha^2 & 0
	\end{bmatrix},  && \mathfrak{M}_{E(4, 5)} = \left[ \begin{array}{cc>{\columncolor{black!15}}ccc} 
		0 & -bc & 0 & -\alpha c & 0 \\ 
		bc & 0 & 0 & 0 & -\alpha b \\
		\rowcolor{black!15}
		0 & 0 & 0 & 0 & 0\\
		\alpha c & 0 & 0 & 0 & -\alpha^2 \\
		0 & \alpha b & 0 & \alpha^2 & 0
	\end{array} \right].
\end{align*}
	Here we highlighted the added $0$-rows/columns in gray.
	We observe by considering the lower right $3 \times 3$ matrices of 
	$\mathfrak{M}_{E(3, 4)}, \mathfrak{M}_{E(3, 5)}, \mathfrak{M}_{E(4, 5)}$ that these matrices
	are linearly independent. 
	This we generalise in the theorem below.
\end{Example}

\begin{Theorem}\label{prop:gensol}
	Let $A\in \mathbb{Z}^{m \times n}$ as above. 
	Then for distinct $i,j \in R$ we have 
	\[
		A^T \cdot \mathfrak{M}_{E(i,j)} = 0.
	\]
	Furthermore, if $A$ is of full rank and $F$ is chosen such that
	the submatrix $A_F$ yields the rank, 
	then the matrices $\mathfrak{M}_{E(i,j)}$ are linearly independent.
\end{Theorem}
\begin{proof}
	By construction, for $s \in R\backslash\{i,j\}$ the $s$-th column of
	$ \mathfrak{M}_{E(i,j)}$ contains nothing but zeros. Hence for arbitrary
	$r\in [m]$, we have 
	\begin{equation}
		\left[ A^T \cdot \mathfrak{M}_{E(i,j)} \right]_{r,s} = 0. \label{eq:proofconstr}
	\end{equation}
	Now let $s \in E(i,j)$. Then
	\[
		\sum_{k=1}^m a_{k,r} \mathfrak{m}_{k,s} = \sum_{k \in E(i,j)} a_{k,r} m_{k,s}=0,
	\]
	by Lemma~\ref{prop:sol}.
	Without loss of generality, assume $i < j$ and  
	$F=[n]$. Then by assumption on the rank, 
	$\beta := (-1)^{i+j} \det \left( A_{[n]} \right) \neq 0$
	and by construction, $\mathfrak{M}_{E(i,j)}$ is of the form as in Figure~\ref{fig:thm}.
	Then $\pm \beta$ is the only entry of the submatrix of $A$ indexed by $[n+1,m] \times [n+1,m]$. 
	This immediately provides the linear independence.
\end{proof}

\begin{figure}[!ht] \label{fig:thm}
 	\[
		\begin{blockarray}{cc|ccccccc}
			& 1\, \cdots \, n& n+1 & \cdots & i & \cdots & j & \cdots & m\\ 
			\begin{block}{c[c|ccccccc]}
			1 & & & & & & & & \\  
			\vdots & \Big{$\ast$} & & & & \Big{$\ast$} & & & \\  
			n & & & & & & & & \\\cline{1-9}
			n+1    & & 0 & \cdots & 0 & \cdots & 0 & \cdots & 0 \\  
			\vdots & & \vdots & \ddots& & & & & \vdots \\  
			i      & & 0 & & 0 & &\beta & & 0 \\
			\vdots & \Big{$\ast$} & \vdots & & & \ddots & & & \vdots\\
			j      & & 0 & & -\beta & & 0 & & 0 \\
			\vdots & & \vdots & & & & & \ddots & \vdots \\
			m      & & 0 & \cdots & 0 & \cdots & 0 & \cdots & 0 \\  
			\end{block}
		\end{blockarray}.
	\]
	\caption{An example of the form of enhanced solution matrices}
\end{figure}

As an immediate consequence we obtain that there are at least 
$\binom{m-n}{2}$ many $m\times m $ integer matrices 
$M$ satisfying 
\[
	A^T \cdot M = [ 0 \; 0 ].
\] 
Together with the final remark from Subsection~\ref{subsec:existence},
we can thus conclude that the above constructed matrices
form a basis of the homogeneous equation~\eqref{eq:homog}.

%%%%%%%%%%%%%%%%%%%%%%%%%%%%%%%%%%%%%%%%%%%%%%%%%%%%%%%%%%
%%%%%%%%%%%%%%%%%%%%%%%%%%%%%%%%%%%%%%%%%%%%%%%%%%%%%%%%%%
\section{Conclusion}
%%%%%%%%%%%%%%%%%%%%%%%%%%%%%%%%%%%%%%%%%%%%%%%%%%%%%%%%%%
%%%%%%%%%%%%%%%%%%%%%%%%%%%%%%%%%%%%%%%%%%%%%%%%%%%%%%%%%%

When does a quantisation for a given cluster algebra $\mathcal{A}(\tilde{B})$ exist and how unique is it? 

The answer we have seen above: it depends on the rank of $\tilde{B}$ and the number of frozen indices.
If the rank of $\tilde{B}$ is small then no quantisation exists. 
On the other hand, if $\tilde{B}$ is of full rank, we distinguish two cases. 
If there is no or only one frozen index, the quantisation is essentially unique. 
Otherwise, we remarked at the end of Section~\ref{subsec:existence} 
that the solution space of matrices satisfying the 
compatibility equation~\eqref{eq:comp} to a fixed skew-symmetriser  
is a vector space over the rational numbers of dimension 
$\left( \begin{smallmatrix} m-n \\ 2 \end{smallmatrix} \right)$. 
In particular, this space is not empty and it contains at least one solution $\Lambda_0$. 
Since Theorem~\ref{prop:gensol} with $A=\tilde{B}$ together with 
an appropriate indexing set $F$ yields a 
linearly independent set of $\left( \begin{smallmatrix} m-n \\ 2 \end{smallmatrix} \right)$ 
solutions to the homogeneous compatibility equation,
all other solutions $\Lambda$ can be constructed as the sum of $\Lambda_0$ and 
a linear combination of all $\mathfrak{M}_{E(i,j)}$ for $i,j\in [m]$. 
In the special case where the principal part of $\tilde{B}$ is already invertible, 
quantisations of full subquivers with all mutable and two frozen vertices 
yield a basis of the homogeneous solution space.

To construct quantum seeds, it is necessary to have integer solutions $\Lambda$ 
for the compatibility equation. Both $\Lambda$ from Theorem~\ref{thm:existence}
and the enhanced solution matrices $\mathfrak{M}_{E(i,j)}$ have integer entries.
However, they do not generate the semigroup of all integer quantisations in general.

What came as a surprise to us is the simple structure of the matrices $\mathfrak{M}_{E(i,j)}$. 
Their computation only depends on $(n+2)\times(n+2)$ minors of $\tilde{B}$, 
which can be realised with little effort. 
The authors used SAGE in their investigations of the problem and 
the first author makes a complementary website available at~\cite{G}. 
There, one can follow the construction of the matrices above in detail, 
compute a general solution to~\eqref{eq:comp} and a generating set of matrices to~\eqref{eq:homog}. 

%%%%%%%%%%%%%%%%%%%%%%%%%%%%%%%%%%%%%%%%%%%%%%%%%%%%%%%%%%

%%%%%%%%%%%%%%%%%%%%%%%%%%%%%%%%%%%%%%%%%%%%%%%%%%%%%%%%%%


\begin{thebibliography}{BIKR}
\providecommand{\url}[1]{{\tt #1}}
\providecommand{\urlprefix}{URL }
\providecommand{\eprint}[2][]{\url{#2}}

\bibitem[BFZ3]{BFZ3}
A.~Berenstein, S.~Fomin, A.~Zelevinsky.
\newblock {\em Cluster algebras. {III}. {U}pper bounds and double {B}ruhat
  cells\/}.
\newblock Duke Math. J., {\bf 126} (2005)~(1), 1--52.
\newblock \urlprefix\url{http://dx.doi.org/10.1215/S0012-7094-04-12611-9}.

\bibitem[BZ]{BZ}
A.~Berenstein, A.~Zelevinsky.
\newblock {\em Quantum cluster algebras\/}.
\newblock Adv. Math., {\bf 195} (2005)~(2), 405--455.
\newblock \urlprefix\url{http://dx.doi.org/10.1016/j.aim.2004.08.003}.

\bibitem[BIKR]{BIKR}
I.~Burban, O.~Iyama, B.~Keller, I.~Reiten.
\newblock {\em Cluster tilting for one-dimensional hypersurface
  singularities\/}.
\newblock Adv. Math., {\bf 217} (2008)~(6), 2443--2484.
\newblock \urlprefix\url{http://dx.doi.org/10.1016/j.aim.2007.10.007}.

\bibitem[CC]{CC}
P.~Caldero, F.~Chapoton.
\newblock {\em Cluster algebras as {H}all algebras of quiver
  representations\/}.
\newblock Comment. Math. Helv., {\bf 81} (2006)~(3), 595--616.
\newblock \urlprefix\url{http://dx.doi.org/10.4171/CMH/65}.

\bibitem[C]{C}
A.~Cayley.
\newblock {\em Sur les d\'{e}terminants gauches\/}.
\newblock {\bf 38} (1849), 93--96.

\bibitem[FST1]{FST1}
S.~Fomin, M.~Shapiro, D.~Thurston.
\newblock {\em Cluster algebras and triangulated surfaces. {I}. {C}luster
  complexes\/}.
\newblock Acta Math., {\bf 201} (2008)~(1), 83--146.
\newblock \urlprefix\url{http://dx.doi.org/10.1007/s11511-008-0030-7}.

\bibitem[FT2]{FT2}
S.~Fomin, D.~Thurston.
\newblock {\em {Cluster algebras and triangulated surfaces. Part II: Lambda
  lengths}\/}.
\newblock  (2012), 76.
\newblock \eprint{1210.5569}, \urlprefix\url{http://arxiv.org/abs/1210.5569}.

\bibitem[FZ1]{FZ1}
S.~Fomin, A.~Zelevinsky.
\newblock {\em Cluster algebras. {I}. {F}oundations\/}.
\newblock J. Amer. Math. Soc., {\bf 15} (2002)~(2), 497--529 (electronic).
\newblock \urlprefix\url{http://dx.doi.org/10.1090/S0894-0347-01-00385-X}.

\bibitem[FZ2]{FZ2}
S.~Fomin, A.~Zelevinsky.
\newblock {\em Cluster algebras. {II}. {F}inite type classification\/}.
\newblock Invent. Math., {\bf 154} (2003)~(1), 63--121.
\newblock \urlprefix\url{http://dx.doi.org/10.1007/s00222-003-0302-y}.

\bibitem[FZ4]{FZ4}
S.~Fomin, A.~Zelevinsky.
\newblock {\em Cluster algebras. {IV}. {C}oefficients\/}.
\newblock Compos. Math., {\bf 143} (2007)~(1), 112--164.
\newblock \urlprefix\url{http://dx.doi.org/10.1112/S0010437X06002521}.

\bibitem[GLS]{GLS}
C.~Gei{\ss}, B.~Leclerc, J.~Schr{\"o}er.
\newblock {\em Cluster structures on quantum coordinate rings\/}.
\newblock Selecta Math. (N.S.), {\bf 19} (2013)~(2), 337--397.
\newblock \urlprefix\url{http://dx.doi.org/10.1007/s00029-012-0099-x}.

\bibitem[GSV]{GSV}
M.~Gekhtman, M.~Shapiro, A.~Vainshtein.
\newblock {\em Cluster algebras and {P}oisson geometry\/}, {\em Mathematical
  Surveys and Monographs\/}, volume 167.
\newblock American Mathematical Society, Providence, RI, 2010.

\bibitem[G]{G}
F.~Gellert.
\newblock {\em Sage functions for the quantisation of cluster algebras\/}.
\newblock
  \urlprefix\url{http://math.uni-bielefeld.de/~fgellert/quantisation.php}.

\bibitem[GY]{GY}
K.~R. Goodearl, M.~T. Yakimov.
\newblock {\em {Quantum cluster algebra structures on quantum nilpotent
  algebras}\/}.
\newblock  (2013), 94.
\newblock \eprint{1309.7869}, \urlprefix\url{http://arxiv.org/abs/1309.7869}.

\bibitem[HL]{HL}
D.~Hernandez, B.~Leclerc.
\newblock {\em Cluster algebras and quantum affine algebras\/}.
\newblock Duke Math. J., {\bf 154} (2010)~(2), 265--341.
\newblock \urlprefix\url{http://dx.doi.org/10.1215/00127094-2010-040}.

\bibitem[K]{K}
D.~E. Knuth.
\newblock {\em Overlapping {P}faffians\/}.
\newblock Electron. J. Combin., {\bf 3} (1996)~(2), Research Paper 5, approx.\
  13 pp.\ (electronic).
\newblock The Foata Festschrift,
  \urlprefix\url{http://www.combinatorics.org/Volume_3/Abstracts/v3i2r5.html}.

\bibitem[La1]{La1}
P.~Lampe.
\newblock {\em A quantum cluster algebra of {K}ronecker type and the dual
  canonical basis\/}.
\newblock Int. Math. Res. Not. IMRN,  (2011)~(13), 2970--3005.
\newblock \urlprefix\url{http://dx.doi.org/10.1093/imrn/rnq162}.

\bibitem[La2]{La2}
P.~Lampe.
\newblock {\em Quantum cluster algebras of type {A} and the dual canonical
  basis\/}.
\newblock Proc. London Math. Soc.,  (2014)~(108), 1--43.
\newblock \urlprefix\url{http://arxiv.org/abs/1101.0580}.

\bibitem[Le]{Le}
B.~Leclerc.
\newblock {\em Dual canonical bases, quantum shuffles and {$q$}-characters\/}.
\newblock Math. Z., {\bf 246} (2004)~(4), 691--732.
\newblock \urlprefix\url{http://dx.doi.org/10.1007/s00209-003-0609-9}.

\bibitem[Lu]{Lu}
G.~Lusztig.
\newblock {\em Introduction to quantum groups\/}, {\em Progress in
  Mathematics\/}, volume 110.
\newblock Birkh\"auser Boston Inc., Boston, MA, 1993.

\bibitem[Z]{Z}
A.~Zelevinsky.
\newblock {\em Private communication\/}.

\end{thebibliography}
\end{document}